\documentclass[11pt]{amsart}
\usepackage{graphicx}
\usepackage{epstopdf}
\usepackage{url}
\usepackage{latexsym}
\usepackage{amsfonts,amsmath,amssymb}
\newtheorem{theorem}{Theorem}[section]
\newtheorem{proposition}[theorem]{Proposition}
\newtheorem{lemma}[theorem]{Lemma}

\newtheorem{corollary}[theorem]{Corollary}

\def\cal{\mathcal}

\date{\today}

\begin{document}

\title[Balanced vertices in  rooted  trees]{Balanced vertices in labeled rooted trees}

\author{Mikl\'os B\'ona}
\address{Department of Mathematics, University of Florida, $358$ Little Hall, PO Box $118105$,
Gainesville, FL, $32611-8105$ (USA)}
\email{bona@ufl.edu}

\begin{abstract}
In a rooted tree, we call a vertex {\em balanced} if it is at equal distance from all its descendant leaves. We
count balanced vertices in three different tree varieties. For decreasing binary trees, we can prove that the probability
that a vertex chosen uniformly at random from the set of all trees of a given size is balanced is monotone decreasing. 
\end{abstract}

\maketitle

\section*{Introduction}
Various parameters of many models of random rooted trees are fairly well understood {\em if they relate to a near-root part of the tree or to global tree structure}.
 The first group includes the numbers of vertices at given distances from the root, the immediate progeny sizes
for vertices near the top, and so on. See \cite{flajolet} for a comprehensive treatment of these results.

Not surprisingly, the technical details of fringe analysis become quite complex as soon
as the focus shifts to layers of vertices further away from the leaves. So while there are
explicit results on the (limiting) fraction of vertices at a fixed, small distance from the
leaves, an asymptotic behavior of this fraction, as a function of the distance, remained an open problem.
Recently, Boris Pittel and the present author have studied this family of questions in \cite{pittel}. 
Most work in fringe analysis focused on {\em decreasing binary trees}, which are also called {\em binary search trees}. We will explain the reason for that. However,
in this paper, we will discuss questions that we can successfully investigate for other tree varieties as well.

We call a vertex $v$ of a rooted tree {\em balanced} if all descending paths from $v$ to a leaf have the same length. In other
words, if $\ell$ is any leaf that is a descendant of $v$, then the unique path from $v$ to $\ell$ consists of $k$ edges, 
where $k$ does not depend on the choice of $\ell$. The number $k$ is called the {\em rank} of $v$.

\section{Decreasing binary trees}
A decreasing binary tree on vertex set $[n]=\{1,2,\cdots ,n\}$ is a binary plane tree in which every vertex has a smaller label
than its parent. This means that the root must have label $n$, every vertex has at most two 
children, and that every child $v$ is either an left child or a right child of its parent, even if $v$ is the {\em only} child of its parent.

Decreasing binary trees on vertex set $[n]$ are in bijection with permutations of $[n]$. In order to see this, 
let $\pi=\pi_1\pi_2\cdots \pi_n$ be a permutation. The  decreasing binary tree
 of $\pi$, which
we denote by $T(\pi)$, is defined as follows. The root of $T(\pi)$ is a vertex
 labeled $n$, the largest entry of $\pi$. 
If $a$ is the largest entry of $\pi$ on the left of $n$, and $b$ is the largest
entry of $\pi$ on the right of $n$, then the root will have two children,
the left one will be labeled $a$, and the right one will  be labeled $b$. If $n$ is
the first (resp. last) entry of $\pi$, then the root will have only one child,
and that is a left  (resp. right) child, and it will necessarily be
labeled $n-1$ as $n-1$ must be the largest of all remaining elements.
Define the rest of $T(\pi)$ recursively,  by taking $T(\pi')$ and $T(\pi'')$, where
$\pi'$ and $\pi''$ are the substrings of $\pi$ on the two sides of $n$, and affixing
them to $a$ and $b$. See Figure \ref{permtree} for an illustration. 

\begin{figure} \label{permtree}
\begin{center}
  \includegraphics[width=60mm]{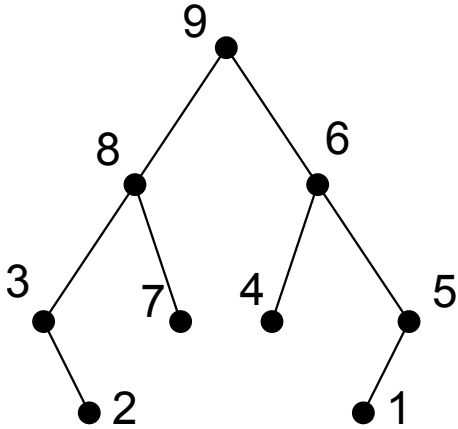}
  \caption{The tree $T(\pi)$ for $\pi=328794615$.}
  \end{center}
 \end{figure}

Note that in the tree $T(\pi)$ shown in Figure \ref{permtree}, all vertices are balanced, except 8, 9, and 6. Also note that 
in any decreasing binary tree, vertex $v$ can be balanced only if all its descendants are balanced. 

\subsection{Balanced vertices of a fixed rank} \label{fixedsection} Recall that a vertex $v$ is {\em balanced} if all descending
paths from $v$  to a leaf have the same length. That length is the rank of $v$.
Let $a_{n,k}$ be the total number of all balanced vertices of rank $k$ in all decreasing binary trees of size $n$. 
Let $A_k(x)=\sum_{n\geq 0} a_{n,k}\frac{x^n}{n!}$, and  let $r_{n,k}$ be the total number of such trees on $n$ vertices
whose {\em root} is balanced of rank $k$. 

\begin{proposition}
The differential equation
\begin{equation} \label{recurrence} 
A_k'(x)=2\frac{A_k(x)}{1-x} + R_k'(x)\end{equation} 
holds, with the initial condition $A_k(0)=0$.  
\end{proposition}

\begin{proof}
Note that $a_{n,k}$ is the number of ordered pairs $(v,T)$, where $v$ is a balanced vertex of rank $k$ in a decreasing
binary tree $T$ on vertex set $[n]$, in other words, $a_{n,k}$ is the number of decreasing binary trees on $[n]$ with a
balanced vertex of rank $k$ marked. If the marked vertex $v$ is not the root of $T$, then removing the root of $T$, we
get, on the one hand, a structure on $[n-1]$ that is counted by $A_k'(x)$, and, on the other hand, a decreasing binary
tree and a decreasing binary tree with a balanced vertex of rank $k$ marked. By the Product formula of exponential generating functions, these pairs of trees are counted by the generating function $2 \cdot \frac{1}{1-x} A_k(x)$. 
The factor 2 is needed since the order of the obtained two trees matters, and $1/(1-x)$ is just the generating function 
of the sequence of factorials, hence of the sequence enumerating decreasing binary trees. Finally, if the marked vertex $v$
is the root of tree, then removing it we just get a structure enumerated by $R_k'(x)$. 
\end{proof}

The special case of $k=0$ counts {\em leaves}, which are of rank 0, and are trivially balanced. Indeed, we have
$R_0(x)=x$, so $R_0'(x)=1$. Therefore \eqref{recurrence} reduces to 
\[A_0'(x)=2\frac{A_0(x)}{1-x} + 1,\] with $A_0(0)=0$. The solution of this differential equation is 
indeed \[A_0(x)=\frac{1}{3} \cdot \frac{1}{(1-x)^2} + \frac{x}{3} -\frac{1}{3},\]
which is the generating function for the number of all leaves of all decreasing binary trees on $n$ vertices. 

If $k=1$, then $R_1(x)=x^2+ \frac{x^3}{3}$, since both trees on two vertices, and both trees on three vertices in which
the root has two children, have a root that is balanced of rank 1. So \eqref{recurrence} reduces to 
 \[A_1'(x)=2\frac{A_1(x)}{1-x} + 2x+ x^2,\] with $A_1(0)=0$. The solution of this initial value problem is
\[A_1(x)=\frac{\frac{1}{5}x^5 - x^3 +x^2}{(1-x)^2}.\] 
Further generating functions $A_k(x)$ can theoretically be computed, but one runs out of computing power very fast. 
A crucial difference from earlier work such as \cite{pittel} is that for all $k$, the generating function $R_k$ is a {\em polynomial function}, since if the root of a tree is balanced and is of rank $k$, then that tree cannot have more than $2^k-1$ vertices. We are now going to show that this implies that
for all $k$, the generating function $A_k(x)$ is always a {\em rational function} of denominator $(1-x)^2$. 
\begin{corollary} Let $k$ be a fixed nonnegative integer.  Let $c_{n,k}=\frac{a_{n,k}}{n! \cdot n}$ be the probability that a  vertex chosen uniformly from the set of all $n\cdot n!$ vertices of all decreasing binary trees on $[n]$ is balanced, and is of rank $k$. Then for any fixed $k$, the limit 
\[c_k=\lim_{n\rightarrow \infty} c_{n,k} \]
exist. Furthermore,  Let $A_k(x)=\frac{P_k(x)}{(1-x)^2}$, that is, let $P_{k}(x)$ be the numerator of $A_k(x)$.
Then  $c_k=P_k(1)$.
\end{corollary}

Note that the fact that the limits $c_k$ exist can also be proved using the techniques of {\em additive functionals} as
explained in \cite{holmgren}. (The number of balanced vertices in a rooted tree is an additive functional.)  However, we provide a self-contained proof here.

\begin{proof}
Solving the linear differential equation (\ref{recurrence}) for $A_k(x)$, we get that
\begin{equation} \label{diffeqsol} A_k(x)=\frac{\int (1-x)^2B_k'(x) \ dx }{(1-x)^2} ,\end{equation}
where the constant of integration is to be chosen so that the initial condition $A_k(0)=0$ is satisfied. 
Recall that $B_k(x)$, and therefore, $B_k'(x)$, is a polynomial function. Therefore, the numerator $P_k(x)$ of the right-hand side
of (\ref{diffeqsol}) is a polynomial function. 

Now let $P_k(x)=P(x)(1-x)^2 + Q(1-x)+R$, where $P(x)$ is a polynomial, and $Q$ and $R$ are complex numbers. 
Note that $P_k(1)=R$. Then
\[A_k(x)=\frac{P_k(x)}{(1-x)^2} = P(x) +\frac{Q}{1-x} +\frac{R}{(1-x)^2}.\] If $n$ is larger than the degree of the
polynomial $P$, then equating coefficients of $x^n$ on both sides, we get that
\[\frac{a_{n,k}}{n!}= Q + (n+1)R.\]
So \[\lim_{n\rightarrow \infty} c_{n,k} = \lim_{n\rightarrow \infty} \frac{a_{n,k}}{n! \cdot n} = \lim_{n\rightarrow \infty} \frac{a_{n,k}}{(n+1)!}= R=P_k(1). \]
\end{proof}

Note that it is easy to see that $c_k>0$ for all $k$. Indeed, $c_k$ is certainly at least as large as the probability that a 
randomly selected vertex is of rank $k$ and is the root of a {\em perfect} binary tree (one in which every non-leaf vertex has exactly 
two children), and even this last probability stays above a positive constant as $n$ goes to infinity. See \cite{protected} 
for details. 

The simple form of $A_k(x)$ is the reason that decreasing binary trees are easier to analyze from this aspect than other tree varieties.
Using the above corollary, we get that $c_0=1/3$, $c_1=1/5$,  $c_2=52/567$ and $c_3=7175243/222660900$.
This shows that for large $n$,  about 65.7 percent of all vertices of decreasing binary trees are balanced and of rank at most three. 
More computation shows that for $n$ sufficiently large, about  66.62 percent of all vertices are balanced and of rank at most four, and about 66.84 percent are
balanced and of rank at most five. 

\subsection{A result about monotonicity}
As decreasing binary trees are in bijection with permutations, we have additional tools analyzing them. This enables 
to prove the following strong result. We could not prove similar results for other labeled rooted trees.

\begin{theorem} \label{decreasing} Let $P_n$ be the probability that a vertex chosen uniformly at random from the set of all vertices of all decreasing binary trees on $[n]$  is balanced.  Then the sequence $P_1,P_2,\cdots $ is weakly decreasing.
\end{theorem}

Let $p_{n,k}$ be  the probability that the {\em  root }of a randomly selected tree on $n$ vertices is balanced, and is of rank $k$. Set $p_{0,i}=1$ for all $i$.  We start by an inequality for the numbers $p_{n,k}$ for fixed $k$. 

\begin{lemma} \label{fixedk}
For all $n\geq 1$ and all fixed $k\leq n$, the inequality $p_{n+1,k} \leq p_{n,k}$ holds. 
\end{lemma}

\begin{proof}
We prove the statement by induction on $n$. The statement is true for all $k$ if $n\leq 3$, since in that case, $p_{n,k}=1$
for all $n$ and all $k$. Now let us assume that the statement is true for $n$ and prove it for $n+1$. 

Let $\pi$ be a permutation of length $n+1$. The probability that the largest entry of $\pi$ is in position $i+1$ for 
any $i\in [0,n]$ is $1/(n+1)$.  The root of $T(\pi)$ is balanced of rank $k$ if and only if all its children are balanced of
rank $k-1$, so 
\begin{equation} \label{recfor(n+1)}
p_{n+1,k} =\frac{ \sum_{i=0}^{n} p_{i,k-1}p_{n-i,k-1}}{n+1} .
\end{equation}
Replacing $n+1$ by $n$, we get the analogous formula

\begin{equation} \label{recfor(n)}
p_{n,k} =\frac{\sum_{i=0}^{n-1} p_{i,k-1}p_{n-1-i,k-1} } {n} .
\end{equation}

The initial difficulty is that while the summands in (\ref{recfor(n+1)}) are smaller than their counterparts in 
(\ref{recfor(n)}), there is one more of them. 
The crucial observation is that if we remove the {\em smallest} summand from the numerator of (\ref{recfor(n+1)}), then
the remaining $n$ summands of that numerator can be matched with the $n$ summands of the numerator of (\ref{recfor(n)}), so that
in each pair, the summand coming from  (\ref{recfor(n)}) is at least as large as the summand coming from  (\ref{recfor(n+1)}). That will prove that the  numerator of (\ref{recfor(n+1)}) is at most $(n+1)/n$ times as large as
the numerator of (\ref{recfor(n)}).

Let $j$ be the index for which $p_{j,k-1}p_{n-j,k-1}$ is minimal, so that last product is the minimal summand in the numerator of (\ref{recfor(n+1)}). First we look at indices smaller than $j$. 
Note that by the induction hypothesis, $p_{n-i,k-1} \leq p_{n-1-i,k-1}$, so 
\begin{equation} \label{ineqfori}  p_{i,k-1}p_{n-i,k-1} \leq p_{i,k-1}p_{n-1-i,k-1}. \end{equation} 
Let us sum these inequalities for $0\leq i\leq j-1$, to get 

\begin{equation} \label{first}
 \sum_{i=0}^{j-1} p_{i,k-1}p_{n-i-1,k-1} \leq    \sum_{i=0}^{j-1} p_{i,k-1}p_{n-1-i,k-1} .
\end{equation}

Now we consider indices larger than $j$. Again, by the induction hypothesis,  $p_{i,k-1}     \leq       p_{i-1,k-1} $, so
\begin{equation}         p_{i,k-1}      p_{n-i,k-1}   \leq       p_{i-1,k-1}      p_{n-i,k-1}    .            \end{equation} 
Let us sum these inequalities for $i\in [j+1,n]$, to get 

\begin{equation} \label{second}
 \sum_{i=j+1}^{n-1} p_{i,k-1}p_{n-i,k-1} \leq    \sum_{i=j+1}^{n-1} p_{i-1,k-1}p_{n-i,k-1} .
\end{equation}

Adding (\ref{first}) and (\ref{second}), we get an inequality whose right-hand side agrees with the sum in (\ref{recfor(n)}), 
and whose left-hand side is the sum in (\ref{recfor(n+1)}), {\em except} the summand of the latter indexed by $j$. However, that summand was the {\em smallest} of the $(n+1)$ summands in the sum in (\ref{recfor(n+1)}), which implies that 

\[ np_{n+1,k} = \frac{n}{n+1}  \sum_{i=0}^{n} p_{i,k-1}p_{n-i,k-1} \leq \sum_{i=0}^{n-1} p_{i,k-1}p_{n-1-i,k-1} =np_{n,k}.\]
This is equivalent to our claim. 
\end{proof}

\begin{corollary} \label{roots}  Let $p_n$ be the probability that the root of  a decreasing binary tree on $[n]$   is balanced.
Then $p_n\geq p_{n+1}$.
\end{corollary}

\begin{proof} It follows from our definitions that
$p_n=\sum_{k=1}^{n-1} p_{n,k} $ and $p_{n+1}=\sum_{k=1}^{n} p_{n+1,k} $.  As Lemma \ref{fixedk} shows that
 $p_{n+1,k} \leq p_{n,k}$ for $k\leq n$, the only issue that we must consider is that the sum that provides $p_{n+1}$ has
one more summand than the sum that provides $p_n$. However, this is not a  problem, since for all $n\geq 2$, 
we have $p_{n,n-1} = 2^{n-1}/n!$, while $p_{n+1,n-1}=2^{n-1}/(n+1)!$ and $p_{n+1,n}=2^n/(n+1)!$, so
\[p_{n,n-1} =  \frac{2^{n-1}}{n!} \geq \frac{3\cdot 2^{n-1}}{n+1}= p_{n+1,n-1} + p_{n+1,n} . \]
This inequality, and applying Lemma \ref{fixedk} for all $k\leq n-2$, proves our claim.  
\end{proof}

\begin{proof} (of Theorem \ref{decreasing}.) Induction on $n$, the initial case of $n=1$ being obvious. In order to prove that $P_n\geq P_{n+1}$, note that a random vertex of a tree of size $n$ has $1/n$ probability to be the root. Furthermore, if  $i\in [n-1]$, then there is an $i/n^2$ probability that a random vertex of a random tree is part of the left subtree of that root, and that left subtree has $i$ vertices. Indeed, the left subtree of the root has $i$ vertices if and only if $n$ is in position $i+1$ of the corresponding permutation, and each of the $i$ vertices of that left subtree are equally likely to be chosen. The same argument applies for right subtrees. This proves that
\[P_n=\frac{p_n}{n}+\frac{2\sum_{i=1}^{n-1} iP_i}{n^2}=\frac{p_n+\sum_{i=1}^{n-1} \frac{2iP_i}{n}}{n}.\]

 Therefore, the inequality $P_n\geq P_{n+1}$ is equivalent to the inequality
\begin{equation} \label{bigone} P_n=\frac{p_n+\sum_{i=1}^{n-1} \frac{2iP_i}{n}}{n} 
\geq \frac{p_{n+1}+\sum_{i=1}^{n} \frac{2iP_i}{n+1}}{n+1}  =P_{n+1}.\end{equation}
In order to prove (\ref{bigone}), note that the first equality in (\ref{bigone}) shows that $P_n$ is obtained as the average of the $n$ summands in the numerator of the fraction that is equal to $P_n$. The average value of a set of real numbers does not change if we add the average value of the set to the set as a new element. In this case, that average value is 
$P_n$, proving that 
\begin{equation} \label{formulaforP}  P_n=\frac{p_n+P_n+ \sum_{i=1}^{n-1} \frac{2iP_i}{n}}{n+1}.\end{equation}

So (\ref{bigone}) will be proved if we can show that 
\[P_n =\frac{p_n+P_n+ \sum_{i=1}^{n-1} \frac{2iP_i}{n}}{n+1} \geq  \frac{p_{n+1}+\sum_{i=1}^{n} \frac{2iP_i}{n+1}}{n+1}  =P_{n+1}.\] Noting that $p_n\geq p_{n+1}$ by Corollary \ref{roots}, it suffices to prove that
\[ P_n+ \sum_{i=1}^{n-1} \frac{2iP_i}{n}  \geq \sum_{i=1}^{n} \frac{2iP_i}{n+1} ,\] which simplifies to the inequality
\begin{equation} \label{easier} \frac{2}{n(n-1)} \sum_{i=1}^{n-1} iP_i  \geq P_n. \end{equation}
Finally, (\ref{easier}) holds, because  the right-hand side of (\ref{formulaforP}) grows if we replace $p_n$ by $P_n$, 
(since in any given tree, the root can only be balanced if all vertices are balanced), which leads to the inequality 
\[ P_n -\frac{2}{n+1} P_n \leq  \frac{\sum_{i=1}^{n-1} \frac{2iP_i}{n}}{n+1},\]
which is clearly equivalent to (\ref{easier}). 

\end{proof}

As the sequence $P_1,P_2,\cdots $ is weakly decreasing, its limit $L$ exists. 
Note that $L\geq  \sum_{k=0}^5  c_{n,k} \approx 0.6684$ as we mentioned in Section
\ref{fixedsection}. On the other hand, the number of {\em  balanced} vertices of rank  larger than five is certainly at most
as large as the number of {\em all} vertices of rank larger than five, and it is known \cite{pittel} that the latter is about  0.00125 times 
the total number of vertices. This proves that
\[0.6684 \leq L \leq 0.66965.\]

\section{Non-plane 1-2 trees}  \label{secnonplane} 
In these trees, the vertices are still bijectively labeled by the elements of $[n]$,  each
vertex has a smaller label than its parent, and each non-leaf vertex has one or two children, but "left" and "right" do not matter anymore. See Figure \ref{fig:fivetrees} for an illustration.  It is well known \cite{anacomb} that the number of such trees is the Euler number $E_n$, and that the exponential generating function of the Euler numbers is 
\[y(x)=\sum_{n\geq 0}E_n \frac{x^n}{n!} = \tan x + \sec x.\] See sequence A000111 in the 
{\em On-line Encyclopedia of Integer Sequences} \cite{oeis} for the many occurrences of these numbers in Combinatorics. 
 
\begin{figure}
 \begin{center}
  \includegraphics[width=70mm]{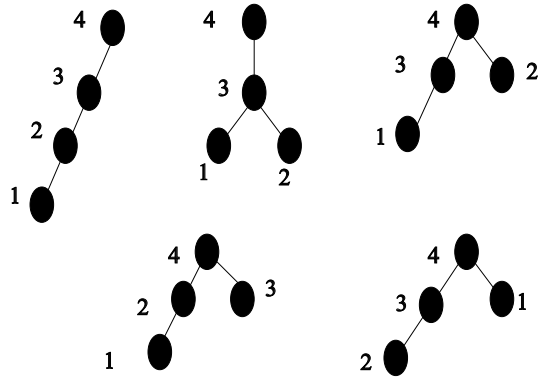}
  \caption{The five rooted non-plane 1-2 trees on vertex set $[4]$. }
  \label{fig:fivetrees}
 \end{center}
\end{figure}

Let $ \cal A_k(x)$ be the exponential generating function 
for the number  of all balanced vertices of rank $k$  in all non-plane 1-2 trees on $[n]$, and let $\cal R_k(x)$ be the exponential generating function for the number of non-plane 1-2 trees on $[n]$  in which {\em the root} is balanced and of rank $k$.

\begin{theorem} \label{eultheo}  The differential equation
 \begin{equation} \label{eulerian} \cal A_k'(x)-\cal A_k(x)y(x) = \cal R_k'(x)\end{equation} 
holds, with the initial condition $\cal A_k(0)=0$. 
\end{theorem}

\begin{proof}
Let $(v,T)$ be an ordered pair in which $T$ is a non-plane 1-2 tree on vertex
set $[n]$ and $v$ is a balanced vertex of rank $k$ of $T$. Then $\cal A_k (x)$ is the exponential 
generating function counting such pairs.  Let us first assume that
$v$ is not the root of $T$,  and let us remove the root of $T$. 
On the one hand, this leaves a structure that is counted by $\cal A_k'(x)$. 
On the other hand, this leaves an ordered pair consisting of a non-plane
1-2 tree with a vertex of order $k$ marked, and a non-plane 1-2 tree. By the Product
formula of exponential generating functions, such ordered pairs are
counted by the generating function $\cal A_k(x) y(x)$. Finally, $v$ was the root of $T$, then
the root of $T$ was balanced and of rank $k$. Such trees are counted by $\cal R_k(x)$, or, after the
removal of their root, by $\cal R_k'(x)$. 
\end{proof}

Crucially, the generating function $\cal R_k(x)$, and hence, its derivative $\cal R_k'(x)$, are {\em polynomials}, which enables us to explicitly solve the linear differential equation (\ref{eulerian}). Indeed, the solution is 
\begin{equation} \label{nonplane} \cal A_k(x)=\frac{\int \cal R_k'(x)(1-\sin x) \ dx}{1-\sin x}, \end{equation}
where the integral in the numerator is an elementary function  since the integral of $x^n\sin x$ is an elementary function for all positive integers $n$. (The constant of integration is chosen so that the initial condition $\cal A_k(0)=0$ is satisfied.) Note that we are not able to count {\em all} verties of rank $k$ in a similar fashion, since the 
generating function for the number of non-plane 1-2 trees in which the root is of rank $k$ is not a polynomial, and 
the solutions analogous to (\ref{nonplane}) will not be elementary functions for $k\geq 1$. Even $\int x \tan x \ dx$ and
$\int x \sec x \ dx$ are not elementary functions. 

\subsection{The number of all vertices} So that we could compute the probability that a randomly selected vertex of a 
randomly selected non-plane 1-2 tree on $[n]$ is balanced and of rank $k$, we need to know the size of the  number $nE_n$  of {\em all} vertices in all such trees. The asymptotics of the Euler numbers are well-known (see \cite{flajolet}, for
example), but to keep the paper self-contained, we provide an argument here at the level of precision that we will need. 
See any introductory textbook on Complex analysis, such as \cite{churchill} for the relevant notions. 

As $y(x)=\tan x + \sec x$, the dominant singularities of $y(x)$ are at $x=\pi /2$ and $x=-\pi/2$, so the coefficients of
$y(x)$ are of exponential order $2/\pi$. (Note that the singularity at $x=-\pi /2$ is removable.)  However, we need a little more more precision. The following proposition will provide that. 

\begin{proposition} \label{residues}
Let $H(x)=f(x)/g(x)$ be a function so that $f(x)$ and $g(x)$ are analytic
functions, $f(x_0)\neq 0$, while $g(x)=0$ and $g'(x)\neq 0$. Then 
\[\mbox{Res H(x)} \mid_{x_0} =\frac{f(x_0)}{g'(x_0)}.\]
\end{proposition}

We can apply Proposition \ref{residues} to $y(x)$ at $x_0=\pi/2$ with $f(x)=1+\sin x$ and $g(x)=\cos x$ if we note that
$y(x)=\frac{1+\sin x}{\cos x}$. Then  Proposition \ref{residues} implies
that $\hbox{Res } y(x) \Bigr|_{\pi/2}=\frac{2}{-1}=-2$. The singularity of $y$ at $x=-\pi/2$ is removable, 
since $\lim_{x\rightarrow -\pi/2} y(x) =0$ exists, so $\hbox{Res } y(x) \Bigr|_{-\pi/2}=0$.

Now observe that
\begin{eqnarray} \frac{R}{x-a} & = & \frac{R}{-a} \cdot \frac{1}{1-\frac{x}{a}}
\\
& = &  \frac{R}{-a} \sum_{n\geq 0} \frac{x^n}{a^n}.
\end{eqnarray} 

Applying this  to $y(x)$ with $a=\pi/2$ and $R=-2$, we get that the dominant
term of $y(x)$ is of the form $\frac{4}{\pi}\sum_{n\geq 0} x^n (2/\pi)^n$, so 
\begin{equation}
\label{eulerprecise}
\frac{E_n}{n!} \sim \frac{4}{\pi} \cdot  \left(\frac{2}{\pi} \right)^n .
\end{equation}

So the total number of all vertices in all non-plane 1-2 trees of size $n$ is 
\begin{equation} \label{eulerprecise1} nE_n \sim  n! \left ( n \frac{4}{\pi}  \cdot   \left(\frac{2}{\pi} \right)^n  \right).\end{equation} 

\subsection{Leaves} \label{leaves}

Let $\cal A_0(x)$ denote exponential generating function for the total number of {\em leaves} in all non-plane 1-2 
trees on vertex set $[n]$. It is then easy to verify that $\cal A_0(x)=x+\frac{x^2}{2}+2\frac{x^3}{6}+
9\frac{x^4}{24}+\cdots $.

\begin{theorem} The equality 
\begin{equation} \label{formofb} \cal A_0(x)=\frac{x-1+\cos x}{1-\sin x}
\end{equation} holds.
\end{theorem}

\begin{proof} Let us apply Theorem \ref{eultheo} with $k=0$, to get the linear differential equation 
\[ \cal A_0'(x)-\cal A_0(x)y(x) = 1.\] Indeed, $R_0(x)=x$, since the only tree whose root is balanced and of rank 0 is
the one-vertex tree. Recalling that $y(x)=\tan x + \sec x$, and the initial condition $\cal A_0(0)=0$, we can solve the
last displayed linear differential equation to get what was to be proved. 
\end{proof}

Note that $x=\pi /2$ is the unique nonremovable singularity of smallest modulus of $\cal A_0(x)$, and that at that point, $\cal A_0(x)$ has 
pole of order two, since $(1-\sin x)' =-\cos x $ also has a zero at that point. Therefore, we cannot apply Proposition 
\ref{residues} directly. Instead, we use the following lemma. 

\begin{lemma} \label{doubleres}
Let $H(x)=\frac{f(x)}{g(x)}$ be a function so that $f$ and $g$ are
analytic functions, $f(x_0)\neq 0$, 
while $g(x_0)=g'(x_0)=0$, and $g''(x)\neq 0$. 
Then 
\[H(x)=\frac{2f(x_0)}{g''(x_0)} \cdot \frac{1}{(x-x_0)^2} + \frac{h_{-1}}{x-x_0}
+h_0+\cdots .\]
\end{lemma}

\begin{proof}
The conditions directly imply that $g$ has a double root, and hence 
$H$ has a pole of order two, at $x_0$. In
order to find the coefficient that belongs to that pole, let 
$g(x)=q(x)(x-x_0)^2$. Now differentiate both sides with respect to $x$, to get
\[g''(x)=q''(x)(x-x_0)^2 + 4q'(x)(x-x_0) + 2q(x).\] Setting $x=x_0$, we
get \begin{equation} \label{fandg} g''(x_0)=2q(x_0).\end{equation}
 By our definitions, in a neighborhood of $x_0$, the function $H(x)$
behaves like
\[\frac{f(x)}{q(x)(x-x_0)^2},\] and our claim follows by (\ref{fandg}).
\end{proof}

Let $\cal A_{n,0}$ be the number of all leaves in all non-plane 1-2 trees on vertex set $[n]$.  

\begin{theorem}
The equality  \[\cal C_0:=\lim_{n\rightarrow \infty} \frac{\cal A_{n,0}}{nE_n} = 1-\frac{2}{\pi} \approx
0.3633802278\] holds. 
In other words, for large $n$, the probability that a vertex chosen uniformly at random from all vertices of all 
non-plane 1-2 trees is a leaf is $1-\frac{2}{\pi}$.
\end{theorem}

\begin{proof}
Note that  $\cal A_0(x)$ has a unique singularity of smallest modulus, 
at $x=\pi/2$, so the exponential growth rate of  its coefficients
is $2/\pi$.  Also note that at that point, 
 the denominator of $\cal A_0(x)$ has a double root. Therefore, Lemma \ref{doubleres}
applies, with $f(x)=x-1+\cos x$ and $g(x)=1-\sin x$, yielding that the coefficient of the $(x-\pi/2)^{-2}$ term in the
Laurent series of $\cal A_0(x)$ about $x_0=\pi/2$  is 
\[2\cdot \frac{(\pi)/2 -1 +\cos (\pi/2)}{\sin(\pi/2)}=\pi-2.\]
Now observe that 
\begin{equation} \label{scaling}  \frac{D}{(x-a)^2} = \frac{D}{a^2} \cdot
 \frac{1}{(1-\frac{x}{a})^2} 
  =  \frac{D}{a^2} \cdot \sum_{n\geq 0}(n+1)\frac{x^n}{a^n}.
\end{equation}
Applying this to the dominant term of $\cal A_0(x)$ with $D=\pi-2$
and $a=\pi/2$, we get that 
\begin{equation} \label{leafprecise}
\frac{\cal A_{n,0}}{n!} \sim n (\pi -2) \cdot \left(\frac{2}{\pi}\right)^{n+2}.
\end{equation}
The proof of our claim is now immediate by comparing formulas 
(\ref{leafprecise}) and (\ref{eulerprecise1}).
\end{proof}

\subsection{Balanced vertices of rank 1}
Let $\cal A_{1}(x)$ be exponential generating function for  the total number of balanced vertices of rank 1 in all non-plane 1-2 trees on
vertex set $[n]$. Note that such vertices have only leaves as neighbors. 

\begin{theorem} The differential equation
\begin{equation} \label{formofc} \cal A_1(x)=\frac{1}{6} \cdot \frac{(3x^2+6x-6)\cos x -(6x+6)\sin x+x^3+3x^2+6}{1-\sin x}
\end{equation} holds, with the initial condition $\cal A_1(0)=0$.
\end{theorem}

\begin{proof} Let us apply Theorem \ref{eultheo} with $k=1$, to get the linear differential equation 
\[ \cal A_1'(x)-\cal A_1(x)y(x) = x+\frac{x^2}{2}.\] Indeed, $R_1(x)=\frac{x^2}{2} + \frac{x^3}{6}$, since there are only two 
 trees whose root is balanced and of rank 1; one of them has two vertices and the other one has three vertices. Recalling that $y(x)=\tan x + \sec x$, and the initial condition $\cal A_1(0)=0$, we can solve the
last displayed differential equation to prove our claim.
\end{proof}

\begin{theorem} \label{theoc1}
The equality  \[\cal C_1 :=\lim_{n\rightarrow \infty} \frac{\cal A_{n,1}}{nE_n} = \frac{\pi^2}{24} + \frac{\pi}{4}-1 \sim 
0.1966316804\]
holds. 
In other words, for large $n$, the probability that a vertex chosen uniformly at random from all vertices of all 
non-plane 1-2 trees is balanced and of rank 1 is $ \frac{\pi^2}{24} + \frac{\pi}{4}-1$.
\end{theorem}

\begin{proof}
Note that $\cal A_1(x)$ has a unique singularity of smallest modulus at $x=\pi/2$, and that that singularity is a pole of order two. Therefore, we can apply Lemma \ref{doubleres} with $f(x)=3x^2+6x-6)\cos x -(6x+6)\sin x+x^3+3x^2+6$ and
$g(x)=6(1-\sin x)$. At $x_0=\pi /2$, this yields $f(x_0)= \frac{\pi^3}{8} + \frac{3\pi^2}{4} -3\pi$. Furthermore, 
$g''(x)=6\sin x$, so $g''(x_0)= 6$. Therefore, the coefficient of the $1/(x-\pi/2)^2$ term in the Laurent series of 
$\cal A_1(x)$ about $x_0=\pi/2$ is 
\[\frac{2f(x_0)}{g''(x)} = \frac{2}{6} \cdot \left (  \frac{\pi^3}{8} + \frac{3\pi^2}{4} -3\pi \right)=
\frac{\pi^3}{24} +  \frac{\pi^2}{4} -\pi .\]

Applying (\ref{scaling}) with $D=\frac{\pi^3}{24} +  \frac{\pi^2}{4} -\pi$ and $a=\pi /2$, we get that 
\begin{equation} \label{almostleafprecise}
\frac{\cal A_{n,1}}{n!} \sim n \left (\frac{\pi^3}{24} +  \frac{\pi^2}{4} -\pi \right ) \cdot \left(\frac{2}{\pi}\right)^{n+2}.
\end{equation}

We can now prove our claim by comparing formulae (\ref{eulerprecise1}) and (\ref{almostleafprecise}). 
\end{proof}

\subsection{Balanced vertices of higher rank}
For any fixed $k$, we can compute the probability that a vertex selected from all vertices of all non-plane 1-2 trees uniformly at random is balanced and of rank $k$. For instance, for $k=2$, we get that 
\[\lim_{n\rightarrow \infty} \frac{\cal A_{n,2}}{nE_n}=
\frac{\pi^6}{32256} +\frac{\pi^5}{2304} -7\frac{\pi^4}{1920}-3\frac{\pi^3}{64} +\frac{\pi^2}{3} +9\frac{\pi}{4}-8
-\frac{2}{\pi} \approx 0.0759013197 ,\]
where $\cal A_{n,2}$ denotes the number of balanced vertices in all non-plane 1-2 trees on $[n]$. 

It is a direct consequence of (\ref{nonplane}), the well-known fact that (easy to prove by induction) 
that \[\int x^n \sin x \ dx = H(x)\sin x +I(x) \cos x,\] for some polynomials $H(x)$ and $I(x)$, and Lemma \ref{doubleres}
that for all positive integers $n$, there exists a polynomial $ J_n$ with rational coefficients so that 
\[\cal C_k:=\lim_{n\rightarrow \infty} \frac{\cal A_{n,k}}{nE_n}= \frac{J_n(\pi)}{\pi} .\]

\section{Plane 1-2 trees}
Plane 1-2 trees are similar to non-plane 1-2 trees, except that the children of each vertex are linearly ordered, left to right.
The difference between plane 1-2 trees and decreasing binary trees is that in plane 1-2 trees, if a vertex has only one 
child, then that child has no "direction", that is, it is not a "left child" or a "right child".  See Figure \ref{plane1-2} for an 
illustration. 

\begin{figure}
 \begin{center}
  \includegraphics[width=70mm]{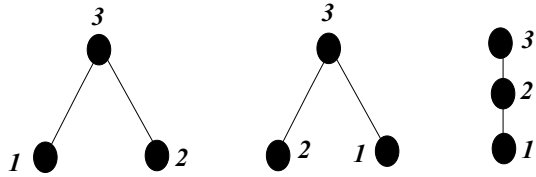}
  \caption{The three rooted plane 1-2 trees on vertex set $[3]$. }
  \label{plane1-2}
 \end{center}
\end{figure}

So plane 1-2 trees are "in between" decreasing binary trees (where left or right matters for every child) and non-plane
1-2 trees (where left or right does not matter for any vertex). Indeed, in plane 1-2 trees, left or right matters, except for
vertices that have no siblings.  

Let $Z(x)=\sum_{n\geq 0}z_n\frac{x^n}{n!}$ be the exponential generating function for the number of plane 1-2 trees on vertex set $[n]$. Then $Z(x)$ satisfies the differential
equation \begin{equation} \label{eqforz} Z'(x)=Z^2(x) - Z(x)+1,\end{equation}
 with initial condition $Z(0)=1$. Indeed, removing the root of a plane 1-2 tree $T$ that has more than one vertex, that tree
 falls apart to the ordered set of two such trees, except when the root of $T$ has only one child. 

See sequence A080635 in \cite{oeis} for many other combinatorial problems whose solution involves the power series
$Z(x)$. 

Solving (\ref{eqforz}), we get the explicit formula
\[Z(x)=\frac{\sqrt{3}}{2} \tan \left ( \frac{\sqrt{3}}{2}x+\frac{\pi}{6} \right  ) +\frac{1}{2} .\]
Noting that $\tan a = \sin a / \cos a$, and that the summand $1/2$ at the end does not influence the growth rate
of the coefficients of $Z(x)$, we can proceed as in Section \ref{secnonplane}. That is, we can use Proposition 
\ref{residues} to compute that the number $z_n$ of plane 1-2 trees on vertex set $[n]$ satisfies
\[z_n \sim n!  \left( \frac{3\sqrt{3}} {2\pi} \right)^{n+1} .\]
     Therefore, the total number of all vertices of all such trees satisfies 
\begin{equation} \label{allvert} nz_n \sim n! \cdot n  \left( \frac{3\sqrt{3}} {2\pi} \right)^{n+1} . \end{equation}

Let $ \mathbf A_k(x)$ be the exponential generating function 
for the number  of all balanced vertices of rank $k$  in all plane 1-2 trees on $[n]$, and let $\mathbf R_k(x)$ be the exponential generating function for the number of plane 1-2 trees on $[n]$  in which {\em the root} is balanced and of rank $k$. 

\begin{theorem} \label{planetheo}  The differential equation
 \begin{equation} \label{plane} \mathbf A_k'(x)=2\mathbf A_k(x)Z(x) - \mathbf A_k(x) + \mathbf R_k'(x)\end{equation} 
holds, with the initial condition $\mathbf A_k(0)=0$. 
\end{theorem}

\begin{proof}
Let $(v,T)$ be an ordered pair in which $T$ is a  plane 1-2 tree on vertex
set $[n]$ and $v$ is a balanced vertex of rank $k$ of $T$. Then $\mathbf A_k (x)$ is the exponential 
generating function counting such pairs.  Let us first assume that
$v$ is not the root of $T$,  and let us remove the root of $T$. 
On the one hand, this leaves a structure that is counted by $\mathbf A_k'(x)$. 
On the other hand, this leaves an order pair consisting of a plane
1-2 tree with a a balanced vertex of rank $k$ marked, and a plane 1-2 tree. If the tree without the marked vertex was
not empty, then by the Product
formula of exponential generating functions, such ordered pairs are
counted by the generating function $2\mathbf A_k(x) Z(x)$, since the order of the two trees matters.
If the tree without the marked vertex was empty, then there is only one way to "order" the 1-element set of 
subtrees of the root, consisting of the subtree with the marked vertex. This results in the correction term 
$-\mathbf A_k'(x)$.

 Finally, $v$ was the root of $T$, then
the root of $T$ was balanced and of rank $k$. Such trees are counted by $\mathbf R_k(x)$, or, after the
removal of their root, by $\mathbf R_k'(x)$. 
\end{proof}

\subsection{Leaves}
Setting $k=0$ in \eqref{plane}, noting that $\mathbf A_0(0)=0$, and $\mathbf R_0'(x)=1$, then solving the resulting differential 
equation we get the following result for the number of all leaves. 
\begin{corollary} The equality 
 \begin{equation} \label{plane0} \mathbf A_0(x) = 
\frac{ \frac{\sqrt{3}}{6}\sin \left(\sqrt{3}x+\frac{\pi}{3} \right )+ \frac{x}{2} -\frac{1}{4} }
{\cos ^2 \left ( \frac{\sqrt{3}x}{2}+\frac{\pi}{6} \right) }
\end{equation} holds.
\end{corollary}

We can apply Lemma \ref{doubleres} to the numerator and the denominator of $\mathbf A_0(x)$ in \eqref{plane0} to  compute the growth rate of the coefficients of that power series. Indeed, $\mathbf A_0(x)$
has a unique singularity of smallest modulus at $x_0=2\pi/(3\sqrt{3})$, and that point the numerator of $\mathbf A_0(x)$
is nonzero,  the denominator, and its first derivative are zero, while the second derivative of the denominator is not 0. 
Then a computation analogous to that immediately following the proof of Theorem \ref{theoc1} shows that 
if $\mathbf a_{n,0}$ is the number of all leaves in all plane 1-2 trees on vertex set $[n]$, then 
\begin{equation} \label{lastleaves} \mathbf a_{n,0} \sim \left ( -\frac{1}{3} + \frac{4}{27}\sqrt{3} \pi \right) \cdot n  \left( \frac{3\sqrt{3}} {2\pi} \right)^{n}.\end{equation}

Comparing \eqref{lastleaves} with \eqref{allvert}, we obtain the following result. 

\begin{theorem} The equality
\[\mathbf C_0: = \lim_{n\rightarrow \infty} \frac{\mathbf a_{n,0}}{nz_n} =\frac{2}{3} - \frac{\sqrt{3}}{2\pi} \approx
0.391002219\] holds. 
\end{theorem}

 In other words, the probability that a vertex selected uniformly at random from all plane 1-2 trees
of size $n$ will be a leaf is about 0.391. 

\subsection{Vertices of higher rank}
Note that remarkably, we can obtain an explicit formula for $\mathbf A_k(x)$ for every $k$. This is in contrast to 
the case when we want to compute the generating function of all vertices of a given rank, where we run into 
non-elementary functions for $k\geq 2$. 

Indeed, the standard form of (\ref{plane}) is 
\[ \mathbf A_k'(x)+(1-2Z(x)) \mathbf A_k(x)=   \mathbf R_k'(x).\]
In order to solve this linear differential equation, first we multiply both sides by the integrating factor 
\begin{eqnarray*} \mu(x) & = & \exp \left ( \int (1-2Z(x)) \ dx \right )= \exp \left ( -\ln \left ( \sec ^2 \left ( \frac{\sqrt{3}x}{2}+\frac{\pi}{6} \right) \right )\right ) \\  & = &
\cos ^2 \left ( \frac{\sqrt{3}x}{2}+\frac{\pi}{6} \right)  ,\end{eqnarray*}
which is an elementary function. 

After that multiplication, we need to integrate both sides of the obtained equation to get $\mathbf A_k(x)$. However, we
are able to do so, since on the right hand side, we have $\mu(x) \mathbf R_k'(x)$, that is, a cosine function times
a {\em polynomial} function, and it is well known that such products have elementary integrals.  

\section{Further directions}
The method that we used in this paper may well be applicable to count balanced vertices in other tree varieties, as long
as the number of children each vertex can have is {\em bounded}. 

It seems intuitively very likely that in any tree variety,   as $n$ goes to infinity, the probability that a vertex chosen 
uniformly at random from all vertices of all trees of size $n$ is balanced will be monotone decreasing. Still in the one case
where we could prove this, the case of decreasing binary trees, our proof heavily depended on the simple bijection between
these trees and permutations. New ideas are needed for other tree varieties. 

We mentioned that plane 1-2 trees are "in between" the other two tree varieties studied in this paper. So it is perhaps interesting that their vertices are {\em the most likely} to be leafs. Indeed, a random vertex of a decreasing binary
tree has a one-third chance to be a leaf, while the same probability is about 36.3 percent for non-plane 1-2 trees, and 
39.1 percent for plane 1-2 trees. Understanding this phenomenon could lead to new insights.


\begin{thebibliography}{99}
\bibitem{protected}  M. B\'ona  $k$-protected vertices in binary search trees. {\em Adv. in Appl. Math}. {\bf 53} (2014), 1--11.
\bibitem{pittel} M. B\'ona, B. Pittel, On a random search tree: asymptotic enumeration of vertices by distance from leaves.
{\em J. Appl. Prob.} to appear. Preprint available at at  arXiv:1412.2796.
\bibitem{anacomb} {\sc   B\'ona, M.} (2016) Introduction to Enumerative and Analytic Combinatorics, CRC Press, 2016. 
\bibitem{churchill} J. Brown, R. V. Churchill, Complex Variables and Applications, 9th edition. McGraw-Hill, 2013. 
\bibitem{flajolet} P. Flajolet and R. Sedgewick, Analytic Combinatorics, Cambridge University Press, Cambridge, UK, 2009.
\bibitem{holmgren} S. Janson and C. Holmgren, Limit laws for functions of fringe trees for binary search trees and recursive trees. {\em Electronic J. Probability} 20 (2015), no. 4, 1-51.
\bibitem{oeis} Online Encyclopedia of Integer Sequences, {\em online database}, \url{www.oeis.org}. 
\end{thebibliography}
\end{document}